\documentclass[a4paper,leqno]{amsart}

\usepackage[T1]{fontenc}
\usepackage[utf8]{inputenc}
\usepackage[normalem]{ulem}
\usepackage{amsfonts,amsmath,amssymb,amsthm,mathtools}
\usepackage{lmodern}
\usepackage{tikz-cd}
\usepackage[english]{babel}
\usepackage{microtype,comment}

\theoremstyle{plain}
\newtheorem{theo}{Theorem}[section]
\newtheorem*{theo*}{Theorem}
\newtheorem{prop}[theo]{Proposition}
\newtheorem{lemm}[theo]{Lemma}
\newtheorem{coro}[theo]{Corollary}
\newtheorem{conj}[theo]{Conjecture}
\theoremstyle{definition}

\newtheorem{rema}[theo]{Remark}

\newcommand{\Z}{\mathbf{Z}}
\newcommand{\Q}{\mathbf{Q}}

\newcommand{\Zp}{\mathbf{Z}_p}

\newcommand{\Qp}{\mathbf{Q}_p}

\DeclareMathOperator{\gal}{Gal}
\DeclareMathOperator{\GL}{GL}

\let\epsilon\varepsilon
\let\varepsilon\epsilon

\let\ker\relax
\DeclareMathOperator{\ker}{Ker}
\DeclareMathOperator{\Fil}{Fil}

\DeclareMathOperator{\coker}{Coker}
\DeclareFontFamily{U}{wncy}{}
\DeclareFontShape{U}{wncy}{m}{n}{<->wncyr10}{}
\DeclareSymbolFont{mcy}{U}{wncy}{m}{n}
\DeclareMathSymbol{\Sha}{\mathord}{mcy}{"58} 

\newcommand{\GG}{\mathcal{G}}
\newcommand{\HH}{\mathcal{H}}
\newcommand{\II}{\mathcal{I}}
\newcommand{\LL}{\mathcal{L}}
\newcommand{\OO}{\mathcal{O}}

\newcommand{\XX}{\mathcal{X}}

\newcommand{\Dcrisv}{\mathbf{D}_{\mathrm{cris},v}}
\DeclareMathOperator{\fil}{Fil}
\DeclareMathOperator{\rank}{rank}
\DeclareMathOperator{\coh}{H}
\DeclareMathOperator{\col}{Col}
\newcommand{\Fcyc}{F(\mu_{p^\infty})}

\newcommand{\Fcycn}{F(\mu_{p^n})}
\newcommand{\Fvcycn}{F_v(\mu_{p^n})}
\newcommand{\vp}{\varphi}

\newcommand{\Xloc}{\XX_{\mathrm{loc}}}
\newcommand{\bc}{\mathbf{c}}
\DeclareMathOperator{\sel}{Sel}
\DeclareMathOperator{\loc}{loc}
\DeclareMathOperator{\ord}{ord}

\newcommand{\p}{\mathfrak{p}}

\newcommand{\uI}{\underline{I}}
\newcommand{\uJ}{\underline{J}}
\newcommand{\cO}{\mathcal{O}}
\newcommand{\Dcris}{\mathbf{D}_{\mathrm{cris}}}
\begin{document}
\author[A. Lei]{Antonio Lei}
\address{(Lei) D\'epartement de Math\'ematiques et de Statistiques, Universit\'e Laval, Pavillon Alexandre-Vachon, 1045 Avenue de la M\'edecine, Qu\'ebec, QC, Canada G1V 0A6}
\email{antonio.lei@mat.ulaval.ca}

\author[G. Ponsinet]{Gautier Ponsinet}
\address{(Ponsinet) Max Planck Institut for Mathematics, Vivatsgasse 7, 53111 Bonn, Germany}
\email{gautier.ponsinet@mpim-bonn.mpg.de}

\thanks{The authors' research is supported by the NSERC Discovery Grants Program 05710.}

\title[On the ranks of supersingular abelian varieties]{On the Mordell-Weil ranks of supersingular abelian varieties in
cyclotomic extensions}
\date{\today}
\begin{abstract}
	Let $F$ be a number field unramified at an odd prime $p$ and
	$F_\infty$ be the $\Zp$-cyclotomic extension of $F$. Let $A$ be
	an abelian variety defined over $F$ with good supersingular
	reduction at all primes of $F$ above $p$.  B\"uy\"ukboduk and
	the first named author have defined modified Selmer groups
	associated to $A$ over $F_\infty$. Assuming that the Pontryagin
	dual of these Selmer groups are torsion
	$\Zp[[\gal(F_\infty/F)]]$-modules, we give an explicit sufficient condition for the rank of the
	Mordell-Weil group $A(F_n)$ to be bounded as $n$
	varies.
\end{abstract}
\maketitle
\tableofcontents

\section*{Introduction}

Let $F$ be a number field. For an odd prime number $p$, let
$(F_n)_{n \geq 0}$ be the tower of number fields such that $F_\infty
= \cup_n F_n$ is the $\Zp$-cyclotomic extension and $\gal(F_n/F)\simeq
\Z/p^n\Z$ (see \S
\ref{subsec:cyclotomic}). 

Let $A$ be an abelian variety defined over $F$. By Mordell-Weil's
theorem, the groups $A(F_n)$ of $F_n$-rational points of $A$ are 
finitely generated abelian groups \cite[Appendix II]{MumfordAbVar}. One
may wonder how these groups vary as $n$ varies.

To study the asymptotic growth of their ranks, a strategy (developed by
Mazur~\cite{Mazur72}) goes as follows . One studies the structure as a
$\Zp[[\gal(F_\infty/F)]]$-module of the Selmer group
$\sel_p(A/F_\infty)$ (see \S \ref{subsec:pSel}), as well as its relation
to the Selmer groups $\sel_p(A/F_n)$ over $F_n$ through Galois descent.
One then deduces information about the rational points via the exact
sequence 
\begin{equation}  \label{seq:intro}
0 \rightarrow A(F_n)\otimes\Qp/\Zp \rightarrow \sel_p(A/F_n) \rightarrow
\Sha_p(A/F_n) \rightarrow 0,
\end{equation}
where $\Sha_p(A/F_n)$ is the $p$-primary component of the
Tate-Shafarevitch group of $A$ over $F_n$.

When $A$ has good ordinary reduction at primes above $p$, the Pontryagin
dual of $\sel_p(A/F_\infty)$ is conjectured to be a torsion
$\Zp[[\gal(F_\infty/F)]]$-module (proved by Kato when $A$ is an elliptic
curve defined over $\Q$ and $F/\Q$ is an abelian extension). Under this
conjecture, Mazur's control theorem on the Selmer groups implies that
the rank of the Mordell-Weil group $A(F_n)$ is bounded independently of
$n$.

However, this approach does not work without the ordinarity assumption.
Firstly, Mazur's control theorem does not hold. Secondly, the Pontryagin
dual of the Selmer group $\sel_p(A/F_\infty)$ is no longer expected to
be a torsion $\Zp[[\gal(F_\infty/F)]]$-module. In the case where $A$ is
an elliptic curve with $a_p=0$ and $F=\Q$, Kobayashi~\cite{Kobayashi}
defined two modified Selmer groups (often referred as plus and minus
Selmer groups in the literature) and proved that they do satisfy the two
aforementioned properties, namely the cotorsioness over an appropriate
Iwasawa algebra and a control theorem à la Mazur. As a consequence, one
deduces that the rank of the Mordell-Weil group $A(F_n)$ is bounded
independently of $n$.  See \cite[Corollary~10.2]{Kobayashi}.
Alternatively, it is possible to deduce the same result using Kato's
Euler system in~\cite{Kato04} and Rohrlich's non-vanishing results on
the complex $L$-values of $A$ in \cite{Rohrlich84}. See the discussion
after Theorems~1.19 and~1.20 in \cite{Greenbergintro}.

We note that a different approach was developed by Perrin-Riou
\cite[\S6]{PR90} to study the Mordell-Weil ranks of an elliptic curve
with supersingular reduction and $a_p=0$.  She showed that when a certain
algebraic $p$-adic $L$-function is non-zero and does not vanish at the
trivial character, then the Mordell-Weil ranks of the elliptic curve
over $F_n$ are bounded independently of $n$. Kim \cite{Kimrank} as well
as Im and Kim \cite{ImKim} have generalized this method to abelian
varieties which can have mixed reduction types at primes above $p$.
Furthermore, unlike the present article, they did not assume that $p$ is
unramified in $F$. Following Perrin-Riou's construction, Im and Kim
defined certain algebraic $p$-adic $L$-functions and showed that when
these are non-zero, then the Mordell-Weil ranks of the abelian variety
over $F_n$ are bounded by certain explicit polynomials in $n$.

The goal of the present article is study the Mordell-Weil ranks of an
abelian variety $A$ over $F_n$ under the assumptions that $p$ is
unramified in $F$ and that $A$ is supersingular at all primes of $F$
above $p$. We make use of the signed Selmer groups developed by
B\"uy\"ukboduk and the first named author in~\cite{BLIntegral} (see
\S\ref{subsec:signedSelmer} for a review of their constructions; these
Selmer groups generalize Kobayashi's plus and minus Selmer groups). Our
main result is the following. 
\begin{theo*}[{{Theorem~\ref{thm:key}}}]
	Assume that the Pontryagin dual of the signed Selmer
	groups of $A$ over $F_\infty$ are torsion
	$\Zp[[\gal(F_\infty/F)]]$-modules. There exist an explicit sufficient condition on the Coleman maps attached to $A$ which ensure that the rank of the
	Mordell-Weil group of the dual abelian variety $A^\vee(F_n)$ is bounded as $n$ varies.
\end{theo*}

We show further that when the Frobenius on the Dieudonné module of $A$ at $p$ can be expressed as certain block matrix, then the explicit condition in Theorem~\ref{thm:key} can be verified (see Corollaries~\ref{cor:antidiag} and \ref{cor:antidiagmodp}). We explain at the end of the paper that our result applies to abelian varieties of $\GL_2$-type.

As opposed to the ordinary case, we do not have a direct relation
between the signed Selmer groups and the Mordell-Weil group as in the
exact sequence~\eqref{seq:intro}. However, we may nonetheless obtain
information about the Selmer groups $\sel_p(A/F_n)$ from the signed
Selmer groups via the Poitou-Tate exact sequence and some calculations
in multi-linear algebra.

\subsection*{Acknowledgement} The authors would like to thank Kazim
B\"uy\"ukboduk, Daniel Delbourgo, Eyal Goren, Byoung Du Kim, Chan-Ho Kim, Jeffrey
Hatley and Florian Sprung for answering many of our questions during the
preparation of this paper. We would also like to thank the anonymous referee for their useful comments on an earlier version of the article.
Parts of this work was carried out while the second-named author was a Ph.D. student at Universit\'e Laval.

\section{Supersingular Abelian Varieties and Coleman maps}
In this section, we set some notations and review results of
\cite{BLIntegral} that we shall need.

\subsection{Cyclotomic extension and Iwasawa algebra}
\label{subsec:cyclotomic}
Fix forever an odd prime $p$ and a number field $F$ unramified at $p$.
We also fix $\overline{F}$ an algebraic closure of $F$ and let $G_F =
\gal(\overline{F}/F)$ be the absolute Galois group of $F$. For each
prime $v$ of $F$, we denote by $F_v$ the completion of $F$ at $v$.
Furthermore, we choose an algebraic closure $\overline{F_v}$ of $F_v$ as
well as an embedding $\overline{F} \hookrightarrow \overline{F_v}$ and
set $G_{F_v} = \gal(\overline{F_v}/F_v)$ to be the decomposition
subgroup of $v$ in $G_F$.  We denote by $\OO_{F}$ the ring of integers
of $F$ and $\OO_{F_v}$ the ring of integers of $F_v$ when $v$ is a
non-archimedean prime of $F$.

Let $\mu_{p^n}$ be the group of $p^n$-th roots of unity in
$\overline{F}$ for every $n
\geq 1$ and $\mu_{p^\infty} = \cup_{n \geq 1} \mu_{p^n}$. We
set $\Fcyc = \cup_{n\geq 1} \Fcycn$ to be the
$p^\infty$-cyclotomic extension of $F$ in $\overline{F}$.
We shall write $\GG_\infty =
\gal(\Fcyc/F)$ for its Galois group. The group $\GG_\infty$
is isomorphic to $\Zp^\times$ and so it may be decomposed as $\Delta
\times \Gamma$, where $\Delta$ is cyclic of order $p-1$ and $\Gamma
\simeq \Zp$. For $n \geq 0$, we write $\Gamma_n$ for the unique
subgroup of $\Gamma$ of index $p^n$.  We set $F_\infty = \Fcyc^\Delta$ 
to be the $\Zp$-cyclotomic extension of $F$, and $F_n = F_\infty^{\Gamma_n}$, 
for $n
\geq 0$.

Let $\Lambda = \Zp[[\GG_\infty]]$ be
the Iwasawa algebra of $\GG_\infty$ over $\Zp$. The aforementioned
decomposition of $\GG_\infty$ tells us that $\Lambda =
\Zp[\Delta][[\Gamma]]$. Furthermore, on fixing a topological generator
$\gamma$ of $\Gamma$, we have an isomorphism $\Zp[[\Gamma]]\simeq
\Zp[[X]]$ induced by $\gamma \mapsto X+1$.
For $n \geq 0$, we denote $\Lambda_n = \Lambda_{\Gamma_n} =
\Zp[\Delta][\Gamma/\Gamma_n]$. The previous isomorphism implies that
$\Zp[\Gamma/\Gamma_n] \simeq \Zp[[X]]/(\omega_n(X))$, where $\omega_n(X)
= (X+1)^{p^n} - 1$. For a character $\eta$ on $\Delta$ and a
$\Lambda$-module $M$, let $M^\eta$ be the $\eta$-isotypic component of
$M$, which is given by $e_\eta \cdot M$, where $e_\eta =
\frac{1}{|\Delta|} \sum_{\delta\in \Delta} \eta^{-1}(\delta)\delta$.
Note that $M^\eta$ is naturally a $\Zp[[\Gamma]]$-module.
We will say that a $\Lambda$-module $M$ has rank $r$ if $M^\eta$ has
rank $r$ over $\Zp[[\Gamma]]$ for all character $\eta$ on $\Delta$.

\subsection{Supersingular abelian varieties}
From now on, we fix a $g$-dimensional abelian variety $A$ defined over
$F$ with good supersingular reduction at every prime $v$ of $F$ dividing
$p$, which means that $A$ has good reduction at $v$ and the slope of the Frobenius acting on the Dieudonné module associated to $A$ at $v$ is constant and equal to $-1/2$ (see \S \ref{sec:coleman}).
For all $n \geq 1$, we write $A[p^n]$ for the group of
$p^n$-torsion points in $A(\overline{F})$ and $A[p^\infty] = \cup_n
A[p^n]$.
Let $T = \varprojlim_{\times p} A[p^n]$ be the $p$-adic Tate module of
$A$, which is a free $\Zp$-module of rank $2g$ endowed with a continuous
action of $G_F$ and let $V = T\otimes_{\Zp}\Qp$. For each prime $v$ of
$F$ dividing $p$, $V$ is a crystalline $G_{F_v}$-representation with
Hodge-Tate weights $0$ and $1$, both with multiplicity $g$.
Finally, we denote by $A^\vee$ the dual abelian variety of $A$.

\begin{lemm} \label{lemm:torsion}
	For any prime $v$ of $F$ dividing $p$ and $w$ a prime of $\Fcyc$
	above $v$, the torsion part of $A(\Fcyc_w)$ is finite and the
	group $A(F_{\infty,w})$ has no $p$-torsion (where we denote
	again by $w$ the prime of $F_\infty$ below $w$).
\end{lemm}
\begin{proof}
	The first statement is a theorem of Imai~\cite{Imai}.
	
	Let $k_{F_v}$ be the residual field of $F_v$. By
	\cite[Lemma~5.11]{Mazur72}, the reduction map induces an
	isomorphism of the $p$-torsion points in
	$A(F_v)$ with the $p$-torsion points in $A(k_{F_v})$. 
	Since $A$ is supersingular at $v$, this latter group is
	trivial. Therefore, $A(F_v)$ has no $p$-torsion. 
	Furthermore, since $\Gamma_{v} = \gal(F_{\infty,w}/F_v) \simeq
	\Zp$ is a pro-$p$ group,
	$A(F_{\infty,w})$ has no $p$-torsion.
\end{proof}

\subsection{Iwasawa cohomology}
Let $v$ be a non-archimedean prime of $F$ and $w$ a prime of $F_\infty$
dividing $v$. We set $v_n$ to be the prime of $F_n$ below $w$. For $i 
\geq 0$, the projective limit of the Galois cohomology
groups $\coh^i(F_{n,v_n},T)$ relative to the corestriction maps is 
denoted by $\coh_{\mathrm{Iw}}^i(F_v,T)$.  The structure of these 
$\Zp[[\Gamma]]$-modules is well-known (see \cite[A.2]{PRLivre}).
\begin{prop} \label{prop:structureIw}
	\begin{enumerate}
	\item The groups $\coh_{\mathrm{Iw}}^i(F_v,T)$ are finitely 
		generated $\Zp[[\Gamma]]$-module, trivial if 
			$i\neq\{1,2\}$.
	\item $\coh_{\mathrm{Iw}}^2(F_v,T)$ is a torsion 
		$\Zp[[\Gamma]]$-module.
	\item The rank of $\coh_{\mathrm{Iw}}^1(F_v,T)$ is given by
\[
	\rank_{\Zp[[\Gamma]]} \coh_{\mathrm{Iw}}^1(F_v,T) = \left\{ 
			\begin{array}{lr} 0 & \text{if } v\nmid p, \\
		2g[F_v:\Qp] & \text{if } v \mid p.
	\end{array} \right.
\]
	\item The torsion sub-$\Zp[[\Gamma]]$-module of 
		$\coh_{\mathrm{Iw}}^1(F_v,T)$ is
			isomorphic to $T^{G_{F_\infty}}$.
	\end{enumerate}
\end{prop}

\subsection{Coleman maps and logarithmic matrices}\label{sec:coleman}
Let $v$ be a prime of $F$ dividing $p$. We shall write $f_v=[F_v:\Qp]$. As a $G_{F_v}$-representation,
$T$ admits a Dieudonn\'e module $\Dcrisv(T)$ \cite{BergerLimites}, which
is a free $\OO_{F_v}$-module of rank $2g$ equipped with a Frobenius
after tensoring by $\Qp$ and a filtration of $\OO_{F_v}$-modules $(\fil^i
\Dcrisv(T))_{i \in \Z}$ such that
\[
	\fil^i \Dcrisv(T) = \left\{ \begin{array}{lr} 
		0 & \text{for } i \geq 1, \\
		\Dcrisv(T) & \text{for } i \leq -1.
	\end{array} \right.
\]
We may choose a $\Zp$-basis $\{u_1,\ldots,u_{2gf_v}\}$ of $\Dcrisv(T)$
such that $\{u_1,\ldots,u_{gf_v}\}$ is a basis for $\fil^0 \Dcrisv(T)$.  The matrix of $\vp$ with respect to this basis is of the form
\[
C_{\vp,v}=C_v\left(
\begin{array}{c|c}
I_{gf_v}&0\\ \hline
0&\frac{1}{p}I_{gf_v}
\end{array}
\right),
\]
where $I_{gf_v}$ denote the identity matrix of dimension $gf_v$ and $C_v$ is some matrix inside $\GL_{2gf_v}(\Zp)$. As in \cite[Definition~2.4]{BLIntegral}, we may define for $n\ge1$, 
\begin{equation}\label{eq:defnmatrix}
C_{v,n}=
\left(
\begin{array}{c|c}
I_{gf_v}&0\\ \hline
0&\Phi_{p^n}(1+X)I_{gf_v}
\end{array}
\right)C_v^{-1}
\quad\text{and}\quad
M_{v,n}=\left(C_{\vp,v}\right)^{n+1}C_{v,n}\cdots C_{v,1},
\end{equation}
where $\Phi_{p^n}$ denotes the $p^n$-th cyclotomic polynomial.

Let $w$ be a prime of $\Fcyc$ dividing $v$. For $i \geq 0$, we will
denote by $\coh^i_{\mathrm{Iw},\mathrm{cyc}}(F_v,T)$ the projective
limit of $\coh^i(\Fcycn_{v_n},T)$ relative to the corestriction maps. 

We set $\HH = \Qp[\Delta] \otimes_{\Qp} \HH(\Gamma)$ where $\HH(\Gamma)$
is the set of elements $f(\gamma -1)$ with $\gamma \in \Gamma$ and $f(X)
\in \Qp[[X]]$ is convergent on the $p$-adic open unit disk.
Perrin-Riou's big logarithm map is a $\Lambda$-homomorphism \cite{PR94}
\[
	\LL_{T,v} : \coh^1_{\mathrm{Iw},\mathrm{cyc}}(F_v,T) \rightarrow 
	\HH \otimes \Dcrisv(T),
\]
which interpolates Kato's dual exponential maps \cite[II \S 
1.2]{Kato91}
\[
	\exp^*_{v,n} : \coh^1(\Fvcycn,T) \rightarrow \Fvcycn
	\otimes \fil^0\Dcrisv(T).
\]
In \cite[Theorem 1.1]{BLIntegral}, the big logarithm map is decomposed 
into
\[
	\LL_{T,v} = (u_1,\ldots,u_{2gf_v})\cdot M_v \cdot 
	\begin{pmatrix}
		\col_{T,v,1}\\
		\vdots\\
		\col_{T,v,2gf_v}
	\end{pmatrix},
\]
where $M_v$ a $2gf_v\times 2gf_v$ logarithmic matrix
defined over $\HH$ given by $\displaystyle\lim_{n\rightarrow\infty}M_{v,n}$ and $\col_{T,v,i}$, $i \in \{1,\ldots,2gf_v\}$ are
$\Lambda$-homomorphisms from $\coh^1_{\mathrm{Iw},\mathrm{cyc}}(F_v,T)$
to $\Lambda$.

If $I_v$ is a subset of $\{1,\ldots,2gf_v\}$, we set
\begin{align*}
	\col_{T,I_v} : \coh^1_{\mathrm{Iw},\mathrm{cyc}}(F_v,T) 
	&\rightarrow \prod_{k=1}^{|I_v|}
	\Lambda,\\
	\mathbf{z} & \mapsto (\col_{T,v,i}(\mathbf{z}))_{i\in
	I_v}.
\end{align*}

Let $\underline{I} = (I_v)_{v \mid p}$ be a tuple of sets indexed by the
primes of $F$ dividing $p$ with $I_v\subset\{1,\ldots, 2gf_v\}$.  We set $\II$ to be the set of all such tuples such that $\sum_{v|p}|I_v|=g[F:\Q]$. We write $\underline{I}_0=(I_{v,0})$ where $I_{v,0}=\{1,\ldots gf_v\}$. Given any $\underline{I}\in \II$ and $\mathbf{z}=z_1\wedge\cdots \wedge z_{g[F:\Q]}\in\bigwedge^{g[F:\Q]} \prod_{v|p}\coh^1_{\mathrm{Iw},\mathrm{cyc}}(F_v,T)$, we define 
\[
\col_{T,\uI}(\mathbf{z})=\det(\col_{T,v,i}(z_j))_{i\in I_v,1\le j\le g[F:\Q]}.
\]

For all $n\ge1$, let $H_{v,n}=C_{v,n}\cdots C_{v,1}$, where the matrices $C_{v,i}$ are defined as in \eqref{eq:defnmatrix}. 
Let $H_n$ be the block diagonal matrix where the blocks on the diagonal are given by $H_{v,n}$.
Given a pair $\underline{I}=(I_v),\underline{J}=(J_v)\in\II$, we define $H_{\uI,\uJ,n}$ to be the $(\uI,\uJ)$-minor of $H_{v,n}$.

\begin{prop}\label{prop:vanish2}
	Let $\mathbf{z}\in\bigwedge^{g[F:\Q]} \prod_{v|p}\coh^1_{\mathrm{Iw},\mathrm{cyc}}(F_v,T)$ and $\theta$ a character on $\GG_\infty$ of conductor $p^{n+1}$. The natural image of $\mathbf{z}$ in $e_\theta\cdot \bigwedge^{g[F:\Q]}\prod_{v|p}\frac{\coh^1(F_v(\mu_{p^{n+1}}),V)}{A(F_v(\mu_{p^{n+1}}))\otimes\Qp}$ is zero if and only if
\[
\sum_{\uJ\in \II}\left(H_{\uI_{0},\uJ,n}\col_{T,\uJ}(\mathbf{z})\right)(\theta)=0.
\]
\end{prop}
\begin{proof}
Let us write $\mathbf{z}=z_1\wedge\cdots\wedge z_{g[F:\Q]}$, $\LL_T=\prod_{v|p}\LL_{T,v}$ and $\exp_{n+1}^*=\prod_{v|p}\exp_{v,n+1}^*$. Since the image of $A(F_v(\mu_{p^{n+1}}))\otimes\Qp$ inside $\coh^1(F_v(\mu_{p^{n+1}}),V)$ is precisely the kernel of the dual exponential map $\exp_{v,n+1}^*$, the image of $\mathbf{z}$ in $$e_\theta\cdot \bigwedge^{g[F:\Q]}\prod_{v|p}\frac{\coh^1(F_v(\mu_{p^{n+1}}),V)}{A(F_v(\mu_{p^{n+1}})_{v})\otimes\Qp}$$ is zero if and only if
\[
e_\theta\cdot \wedge_{1\le j\le g[F:\Q]} \exp_{n+1}^*(z_j)=0.
\]
 Via the interpolation formula of Perrin-Riou's big logarithm map (see for example \cite[Theorem~B.5]{LLZp2}), this is equivalent to
\begin{equation}\label{eq:vanish2}
\wedge_{1\le j\le g[F:\Q]} (1\otimes \vp^{-n-1})\circ\LL_{T}(z_j)(\theta)=0.
\end{equation}

For all $1\le j\le g[F:\Q]$, we may write $\LL_{T}(z_j)=\left(\LL_{T,v}(z_j)\right)_{v|p}$, where $\LL_{T,v}(z_j)$ is understood to be the composition of $\LL_{T,v}$ with the projection $\prod_{w|p} \coh^1_{\mathrm{Iw},\mathrm{cyc}}(F_w,T)\rightarrow \coh^1_{\mathrm{Iw},\mathrm{cyc}}(F_v,T)$. Recall from \cite[proof of Proposition~2.5]{BLIntegral}  that
\[
\LL_{T,v}(z_j)\equiv\begin{pmatrix}
u_1&\cdots &u_{2gf_v}
\end{pmatrix}
\cdot M_{v,n}\cdot\begin{pmatrix}
\col_{T,v,1}(z_j)\\ \vdots \\ \col_{T,v,2gf_v}(z_j)
\end{pmatrix}\mod \omega_{n}\HH\otimes\Dcrisv(T).
\]
But the right-hand side is equal to
\[
\begin{pmatrix}
\vp^{n+1}(u_1)&\cdots &\vp^{n+1}(u_{2gf_v})
\end{pmatrix}
 H_{v,n}\begin{pmatrix}
\col_{T,v,1}(z_j)\\ \vdots \\ \col_{T,v,2gf_v}(z_j)
\end{pmatrix},
\]
which implies that
\[
(1\otimes \vp^{-n-1})\circ\LL_{T,v}(z_j)\equiv\begin{pmatrix}
u_1&\cdots &u_{2gf_v}
\end{pmatrix}
\cdot H_{v,n}\cdot\begin{pmatrix}
\col_{T,v,1}(z_j)\\ \vdots \\ \col_{T,v,2gf_v}(z_j)
\end{pmatrix}\mod \omega_{n}\HH\otimes\Dcrisv(T).
\]
On taking wedge products, \eqref{eq:vanish2} is thus equivalent to the vanishing of 
\[
\sum_{\uI,\uJ\in \II}\wedge_{i\in I_v}u_{i}
\left( H_{\uI,\uJ,n}\col_{T,\uJ}(\mathbf{z})\right)(\theta).
\]

It remains to show that $H_{\uI,\uJ,n}$ vanishes at $\theta$ unless $\uI=\uI_0$ and $\uJ\in\II$. Indeed, if $i\in \{gf_v+1,\ldots 2gf_v\}$, the entire $i$-th row of $C_{v,n}$ is divisible by $\Phi_{p^n}(1+X)$. Hence, the same is true for $H_{v,n}$. In particular, when we evaluate it at $\theta$, the whole row becomes zero. In other words, the lower half of $H_{v,n}(\theta)$ is entirely zero. Therefore, in order for a $g[F:\Q]\times g[F:\Q]$ minor to be non-zero, we must take the upper half of $H_{v,n}$. In other words, $\uI=\uI_0$.
\end{proof}

\section{Selmer groups}
In this section, we introduce various Selmer groups associated to
$A^\vee$, gather some of their properties that we shall need and finish
by using the Poitou-Tate exact sequence to relate them to one another.

\subsection{Signed Selmer groups} \label{subsec:signedSelmer}
Let $v$ be a prime of $F$ dividing $p$ and $w$ a prime of $\Fcyc$
above $v$ and fix $\uI=(I_v)\in\II$. We define 
\[
	\coh^1_{I_v}(\Fcyc_w,A^\vee[p^\infty]) \subset\coh^1(\Fcyc_w,A^\vee[p^\infty])
\]
to be the orthogonal complement of $\ker \col_{T,I_v}$ under Tate's
local pairing
\[
	\coh^1(\Fcyc_w,A^\vee[p^\infty]) \times
	\coh_{\mathrm{Iw},\mathrm{cyc}}^1(F_v,T) \rightarrow \Qp/\Zp.
\]
Since $A^\vee[p^\infty](\Fcyc_w)$ is a finite $p$-group by
Lemma~\ref{lemm:torsion} and the order of $\Delta$ is $p-1$, the groups
$\coh^1(\Delta,A^\vee[p^\infty](\Fcyc_w))$ and
$\coh^2(\Delta,A^\vee[p^\infty](\Fcyc_w))$ are trivial. Therefore, by
the inflation-restriction exact sequence, the restriction map 
\[
	\coh^1(F_{\infty,w},A^\vee[p^\infty]) \rightarrow
	\coh^1(\Fcyc_w,A^\vee[p^\infty])^\Delta
\]
is an isomorphism. We use this isomorphism to define
$\coh^1_{I_v}(F_{\infty,w},A^\vee[p^\infty]) \subset
\coh^1(F_{\infty,w},A^\vee[p^\infty])$ by
\[
	\coh^1_{I_v}(F_{\infty,w},A^\vee[p^\infty]) =
	\coh^1_{I_v}(\Fcyc_w,A^\vee[p^\infty])^\Delta .
\]

The $\underline{I}$-Selmer group of $A^\vee$ over $F_\infty$ is then
defined by 
\[
	\sel_{\underline{I}}(A^\vee/F_\infty) = \ker
	\left(\coh^1(F_\infty,A^\vee[p^\infty]) \rightarrow
	\prod_{w \nmid p } \coh^1(F_{\infty,w},A^\vee[p^\infty]) \times
	\prod_{w\mid p }
	\frac{\coh^1(F_{\infty,w},A^\vee[p^\infty])}{\coh^1_{I_v}(F_{\infty,w},A^\vee[p^\infty])}
	\right).
\]
\begin{rema}
	When $F = \Q$ and $A$ is an elliptic curve with $a_p = 0$, for
	an appropriate choice of basis for the Dieudonn\'e module of
	$T$, the signed Selmer groups coincide with
	Kobayashi~\cite{Kobayashi} plus and minus Selmer groups. See \cite[Appendix~4]{BLIntegral}.
\end{rema}

We denote by $\XX_{\underline{I}}(A^\vee/F_\infty)$ the Pontryagin dual of $\sel_{\underline{I}}(A^\vee/F_\infty)$. As in \cite{Kobayashi}, we have the following conjecture.
\begin{conj} \label{conj}
	For all $\underline{I} \in \II$,
	the $\Zp[[\Gamma]]$-module
	$\XX_{\underline{I}}(A^\vee/F_\infty)$
	is torsion.
\end{conj}
When $A$ is an elliptic curve defined over $\Q$, Conjecture~\ref{conj} is known to be true (c.f.
\cite{Kobayashi,Sprung}). See also \cite{LLZWach} where a similar conjecture has been proved for modular forms.
\subsection{$p$-Selmer groups} \label{subsec:pSel}
The $p$-Selmer group of $A^\vee$ over an algebraic extension $K$ of $F$
is defined by 
\[
	\sel_p(A^\vee/K) = \ker \left(\coh^1(K,A^\vee[p^\infty])
	\rightarrow \prod_{v}
	\frac{\coh^1(K_v,A^\vee[p^\infty])}{A^\vee(K_v)\otimes
	\Qp/\Zp}\right),
\]
where the injection $A^\vee(K_v) \otimes \Qp/\Zp \hookrightarrow
\coh^1(K_v,A^\vee[p^\infty])$ is the Kummer map. 
Note that $A^\vee(K_v) \otimes \Qp/\Zp = 0$ when $v$ does not divide
$p$. Furthermore, the orthognal complement of $A^\vee(K_v) \otimes
\Qp/\Zp$ under Tate's local pairing 
\[
	\coh^1(K_v,A^\vee[p^\infty]) \times \coh^1(K_v,T) \rightarrow
	 \Qp/\Zp
\]
is $A(K_v) \otimes \Zp$.
The $p$-Selmer group then fits into a short exact sequence
\begin{equation} \label{eq:sel}
	0 \rightarrow A^\vee(K) \otimes \Qp/\Zp \rightarrow
	\sel_p(A^\vee/K) \rightarrow \Sha(A^\vee/K)[p^\infty]
	\rightarrow 0,
\end{equation}
where $\Sha(A^\vee/K)$ is the Tate-Shafarevich group of $A^\vee$ over
$K$. 
We denote by $\XX_p(A^\vee/K)$ the Pontryagin dual
of $\sel_{\underline{I}}(A^\vee/K)$.

\subsection{Fine Selmer groups}
The fine Selmer group of $A^\vee$ over an algebraic extension $K$ of $F$
is defined by 
\[
	\sel_0(A^\vee/K) = \ker \left(\coh^1(K,A^\vee[p^\infty]) \rightarrow
	\prod_v \coh^1(K_v,A^\vee[p^\infty])\right).
\]
We denote by $\XX_0(A^\vee/K)$ its Pontryagin dual.

One has a
``control theorem'' for the fine Selmer groups in the cyclotomic
extension, which we prove following closely
Greenberg~\cite[\S 3]{GreenbergIwasawaEll} and
\cite{GreenbergGaloisSelmer}.
\begin{lemm} \label{lemm:control}
	The kernel and cokernel of the restriction map 
	\[
		\sel_0(A^\vee/F_n) \rightarrow
		\sel_0(A^\vee/F_\infty)^{\Gamma_n}
	\]
are finite and have bounded orders as $n$ varies.
\end{lemm}
\begin{proof}
	The diagram
	\begin{equation} \label{diag:control}
			\begin{tikzcd}
				0 \arrow{r} & \sel_0(A^\vee/F_n) 
				\arrow{r}
				\arrow{d} & \coh^1(F_n,A^\vee[p^\infty])
				\arrow{r} \arrow{d} & \prod_{v_n}
				\coh^1(F_{n,v_n},A^\vee[p^\infty])
				\arrow{d}\\
				0 \arrow{r} & 
				\sel_0(A^\vee/F_\infty)^{\Gamma_n} 
				\arrow{r}
				& 
				\coh^1(F_\infty,A^\vee[p^\infty])^{\Gamma_n}
				\arrow{r} & \prod_{w}
			\coh^1(F_{\infty,w},A^\vee[p^\infty])^{\Gamma_n}
		\end{tikzcd}
	\end{equation}
	is commutative.

	One has the inflation-restriction exact sequence
	\[
		0 \rightarrow
		\coh^1(\Gamma_n,A^\vee(F_\infty)[p^\infty]) \rightarrow
		\coh^1(F_n,A^\vee[p^\infty]) \rightarrow
		\coh^1(F_\infty,A^\vee[p^\infty])^{\Gamma_n} \rightarrow
		\coh^2(\Gamma_n,A^\vee(F_\infty)[p^\infty]).
	\]
	By Lemma~\ref{lemm:torsion}, the groups
	$\coh^1(\Gamma_n,A^\vee(F_\infty)[p^\infty])$ and
	$\coh^2(\Gamma_n,A^\vee(F_\infty)[p^\infty])$ are trivial. Thus,
	the central vertical map of the diagram is an isomorphism.

	We now study the rightmost vertical map prime by prime. Let $v$ be any
	prime of $F$ and $v_n$ be any prime of $F_n$ above $v$. Let $r_n$ be the
	restriction map 
	\[
		\coh^1(F_{n,v_n},A^\vee[p^\infty]) \rightarrow
		\coh^1(F_{\infty,w},A^\vee[p^\infty]),
	\]
	where $w$ is any prime of $F_\infty$ dividing $v_n$.
	If $v$ is archimedean, then
	$v$ splits completely in $F_\infty/F$. Thus, $\ker(r_n) =0$.
	If $v$ is a non-archimedean prime, by the inflation-restriction exact
	sequence 
	\[
		\ker (r_n) \simeq
		\coh^1(\Gamma_{v_n},A^\vee(F_{\infty,w})[p^\infty]).
	\]
	By Lemma~\ref{lemm:torsion}, if $v$ divides $p$, this last
	group is trival. We assume that $v$ does not divide $p$.
	Then $v$ is unramified and finitely decomposed in $F_\infty$.
	Thus, $F_{\infty,w}/F_v$ is an unramified $\Zp$-extension. Let
	$\gamma_{v_n}$ be a topological generator of $\Gamma_{v_n}$.
	Then
	\[
		\ker (r_n) \simeq
		\coh^1(\Gamma_{v_n},A^\vee(F_{\infty,w})[p^\infty])
		\simeq A^\vee(F_{\infty,w})[p^\infty]/(\gamma_{v_n} -
		1)A^\vee(F_{\infty,w})[p^\infty].
	\]
	As a group $A^\vee(F_{\infty,w})[p^\infty] \simeq (\Qp/\Zp)^t \times
	(\text{a finite group})$, for some $0\leq t \leq 2g$.
	Since $A^\vee(F_{n,v_n})$ is finitely generated, the kernel of
	$(\gamma_{v_n} -1)$ acting on $A^\vee(F_{\infty,w})[p^\infty]$
	is finite. Thus, the restriction of $(\gamma_{v_n} - 1)$ on 
	the maximal divisible subgroup which we write
	$(A^\vee(F_{\infty,w})[p^\infty])_\mathrm{div}$ is surjective
	and we have
	\[
		(A^\vee(F_{\infty,w})[p^\infty])_\mathrm{div} \subset
		(\gamma_{v_n} -1) A^\vee(F_{\infty,w})[p^\infty].
	\]
	Therefore, the cardinality of $\ker (r_n)$ is bounded by
	$[A^\vee(F_{\infty,w})[p^\infty]:(A^\vee(F_{\infty,w})[p^\infty])_\mathrm{div}]$
	which is independent of $n$.
	Furthermore, if $A^\vee$ has good reduction at $v$, then the
	inertia subgroup of $G_{F_v}$ acts trivially and
	$A^\vee(F_{\infty,w})[p^\infty]$ is divisible.
	Hence,
	$\ker(r_n)$ is trivial. 

	Now the set of non-archimedean primes of $F$ where $A^\vee$
	has bad reduction is finite and for each of these primes,
	the order of $\ker(r_n)$ is bounded as $n$ varies and the number
	of primes $v_n$ of $F_n$ dividing $v$ is also bounded, hence,
	the order of the kernel of the right-most vertical map in the
	diagram is bounded as $n$ varies.

	We conclude by applying the snake lemma to the 
	diagram~\eqref{diag:control}.
\end{proof}

\subsection{Poitou-Tate exact sequences}
Let $\Sigma$ be a finite set of primes of $F$ containing the primes
dividing $p$, the archimedean primes and the primes of bad reduction of
$A^\vee$. If $K$ is an extension of $F$, we say by abuse that a prime of
$K$ is in $\Sigma$ if it divides an element of $\Sigma$ and we denote by
$K_\Sigma$ the maximal extension of $K$ unramified outside $\Sigma$.
The cyclotomic extension $\Fcyc$ is contained in $F_\Sigma$ since only
archimedean primes and primes dividing $p$ can ramify in $\Fcyc$.
Furthermore, the action of
$G_F$ on $A^\vee[p^\infty]$ factorizes through $\gal(F_\Sigma/F)$.  In
particular, for $F^\prime$ any extension of $F$ contained in $F_\infty$
and $\ast \in \{p,0,\underline{I}\}$, we have that
$\sel_{\ast}(A^\vee/F^\prime) \subset
\coh^1(F_\Sigma/F^\prime,A^\vee[p^\infty])$. 
Therefore, all the Pontryagin duals $\XX_{\ast}(A^\vee[p^\infty]/F_\infty)$
are finitely generated $\Zp[[\Gamma]]$-modules (see
\cite{GreenbergIwasawaRep}).

For $i \geq 0$, let $\coh^i_{\mathrm{Iw},\Sigma}(F,T)$ be the
projective limit of the groups $\coh^i(F_\Sigma/F_n,T)$ relative to the corestriction maps.
By \cite[Proposition A.3.2]{PRLivre}, we have the exact sequences
\begin{align}
	0 \rightarrow \XX_0(A^\vee/F_\infty) \rightarrow
	\coh^2_{\mathrm{Iw},\Sigma}(F,T) \rightarrow \prod_{w \in
	\Sigma}\coh^2_{\mathrm{Iw}}(F_v,T), \label{diag:finePT} \\ 
	\coh^1_{\mathrm{Iw},\Sigma}(F,T) \rightarrow \prod_{w\in \Sigma}
	\frac{\coh^1_{\mathrm{Iw}}(F_v,T)}{\ker \col_{I_v}} \rightarrow
	\XX_{\underline{I}}(A^\vee/F_\infty) \rightarrow
	\XX_0(A^\vee/F_\infty)
	\rightarrow 0, \label{diag:signedPT} 
\end{align}
and, for any $n \geq 0$,
\begin{align}
	\coh^1(F_{n,\Sigma},T) \rightarrow \prod_{v|p}
	\frac{\coh^1(F_{n,v},T)}{A(F_{n,v})\otimes \Zp} \rightarrow
	\XX_p(A^\vee/F_n) \rightarrow \XX_0(A^\vee/F_n) \rightarrow 0
	\label{diag:pPT}.
\end{align}

\begin{lemm} \label{lemm:PoitouTate}
	Assume that Conjecture~\ref{conj} holds. Then
	\begin{enumerate}
		\item \label{lemm:PoitouTate1} $\XX_0(A^\vee/F_\infty)$ is a torsion
			$\Zp[[\Gamma]]$-module,
		\item \label{lemm:PoitouTate2} $\coh^2_{\mathrm{Iw},\Sigma}(F,T)$ is
			a torsion $\Zp[[\Gamma]]$-module,
		\item \label{lemm:PoitouTate3} $\coh^1_{\mathrm{Iw},\Sigma}(F,T)$ is a
			$\Zp[[\Gamma]]$-module of rank $g[F:\Q]$.
	\end{enumerate}
\end{lemm}
\begin{proof}
	The exact sequence \eqref{diag:signedPT} tells us that
	Conjecture~\ref{conj} implies \eqref{lemm:PoitouTate1}.

	By Propositon~\ref{prop:structureIw}, for every $w \in \Sigma$,
	$\coh^2_{\mathrm{Iw}}(F_v,T)$ is a torsion
	$\Zp[[\Gamma]]$-module. Hence, using the exact
	sequence~\eqref{diag:finePT}, the first statement of the Lemma
	implies the second.

	Finally, thanks to \cite[Proposition 1.3.2 (i) $\Rightarrow$ 
	(ii)]{PRLivre}, we conclude that
	\eqref{lemm:PoitouTate2} implies \eqref{lemm:PoitouTate3}.
\end{proof}

\section{Growth of ranks}

\subsection{Bounding Mordell-Weil ranks using logarithmic matrices}
We define
\[
\Xloc(F_n)=\coker\left(\coh^1_{\mathrm{Iw},\Sigma}(F,T)\rightarrow \prod_{v|p} \frac{\coh^1(F_{n,v},T)}{A(F_{n,v})\otimes\Zp}\right).
\]
\begin{lemm}\label{lem:cokernel}
For $n\gg0$,
\[\rank_{\Zp}\Xloc(F_n)=\rank_{\Zp}\XX_p(A^\vee/F_n)+O(1).\]
\end{lemm}
\begin{proof}
By Lemma~\ref{lemm:control},
\[
\rank_{\Zp}\XX_0(A^\vee/F_n)=\rank_{\Zp}\XX_0(A^\vee/F_\infty)_{\Gamma_n}.
\]
Since Lemma~\ref{lemm:PoitouTate}\eqref{lemm:PoitouTate1} says that $\XX_0(A^\vee/F_\infty)$ is $\Zp[[\Gamma]]$-torsion, we have
\[\rank_{\Zp}\XX_0(A^\vee/F_\infty)_{\Gamma_n}=O(1)\]
for $n\gg0$. Hence, the lemma follows from \eqref{diag:pPT}.
\end{proof}
Therefore, in order to bound the Mordell-Weil rank of $A^\vee(F_n)$, it is enough to bound $\rank_{\Zp}\Xloc(F_n)$ thanks to \eqref{eq:sel}. We explain below how we may obtain a bound on $\rank_{\Zp}\Xloc(F_n)$  using the logarithmic matrices we studied in \S\ref{sec:coleman}.

From now on, we fix a family of classes $c_1,c_2,\ldots, c_{g[F:\Q]}\in \coh^1_{\mathrm{Iw},\Sigma}(F,T)$ such that $\coh^1_{\mathrm{Iw},\Sigma}(F,T)/\langle c_1,\ldots, c_{g[F:\Q]}\rangle$ is $\Zp[[\Gamma]]$-torsion (their existence is guaranteed by Lemma~\ref{lemm:PoitouTate}\eqref{lemm:PoitouTate3}).

	The composition
	\[
		\coh^1_{\mathrm{Iw},\Sigma}(F,T) \xrightarrow{(\loc_v)_v} \prod_{v\mid p}\coh^1_{\mathrm{Iw},\mathrm{cyc}}(F_v,T) \xrightarrow{(\col_{T,J_v})_v} \prod_{k=1}^{g[F:\Q]}\Lambda
	\]
	is a $\Lambda$-homomorphism between two $\Lambda$-modules of rank $g[F:\Q]$.
	Let us write
\[
	\col_{T,\uJ}(\bc)=\det\left(\col_{T,J_v}\circ \loc_v(c_i)\right)_{v|p,1\le i\le g[F:\Q]}.
\]
	
\begin{lemm}\label{lem:nonzero}
Let $\uJ\in \II$. Suppose that the Selmer group $\sel_{\uJ}(A^\vee/F_\infty)$ is $\Zp[[\Gamma]]$-cotorsion. Then, $\col_{T,\uJ}(\bc)\ne 0$.
\end{lemm}
\begin{proof}
Recall from  Lemma~\ref{lemm:PoitouTate}\eqref{lemm:PoitouTate1}  that $\XX_0(A^\vee/F_\infty)$ is $\Zp[[\Gamma]]$-torsion. By assumption,  $\XX_{\uJ}(A^\vee/F_\infty)$ is also $\Zp[[\Gamma]]$-torsion. Therefore, our result follows from \eqref{diag:signedPT}.
\end{proof}

\begin{prop}\label{prop:boundXloc}
Let $\theta$ a character on $\Gamma$ of conductor $p^{n+1}$ which is trivial on $\Delta$. Then, $e_\theta\cdot \Xloc(F_{n})\otimes_{\Zp} \Qp=0$ if
\[
\sum_{\uJ\in\II}\left(H_{\uI_0,\uJ,n}\col_{T,\uJ}(\bc)\right)(\theta)\ne0.
\]
\end{prop}
\begin{proof}
Note that
\[
\frac{\coh^1(F_v(\mu_{p^{n+1}}),V)}{A(F_v(\mu_{p^{n+1}}))\otimes\Qp}\cong F_v[\gal(F_{v}(\mu_{p^{n+1}})/F_v)]^{\oplus g}\cong \Qp[\gal(F_{v}(\mu_{p^{n+1}})/F_v)]^{\oplus gf_v}
\]
as  $\Gamma$-modules via the Bloch-Kato dual exponential map. Thus,
\[
\bigwedge^{g[F:\Q]}\left(\prod_{v|p}\frac{\coh^1(F_v(\mu_{p^{n+1}}),V)}{A(F_v(\mu_{p^{n+1}}))\otimes\Qp}\right)
\cong \Qp[\gal(F(\mu_{p^{n+1}})/F)].
\]
In particular, 
$$e_\theta\cdot  \bigwedge^{g[F:\Q]}\left(\prod_{v|p}\frac{\coh^1(F_v(\mu_{p^{n+1}}),V)}{A(F_v(\mu_{p^{n+1}}))\otimes\Qp}\right)=e_\theta\cdot  \bigwedge^{g[F:\Q]}\left(\prod_{v|p}\frac{\coh^1(F_{n,v},V)}{A(F_{n,v})\otimes\Qp}\right)$$
is a one-dimensional $\Qp(\theta)$-vector space. By Proposition~\ref{prop:vanish2}, our hypothesis on $\wedge c_i$  tells us that its image in
this vector space is non-zero. Therefore, the $e_\theta$-component of the cokernel of
\[
\bigwedge^{g[F:\Q]}\coh^1_{\mathrm{Iw},\Sigma}(F,T)\otimes\Qp \rightarrow \bigwedge^{g[F:\Q]}\prod_{v|p} \frac{\coh^1(F_{n,v},V)}{A(F_{n,v})\otimes\Qp}
\] 
is zero. Hence the result.
\end{proof}

\begin{theo}\label{thm:key}
Let $\theta$ a character  as in the statement of Proposition~\ref{prop:boundXloc}. Suppose that 
\begin{equation}\label{eq:key}
\sum_{\uJ}\left(H_{\uI_0,\uJ,n}\col_{T,\uJ}(\bc)\right)(\theta)\ne0,
\end{equation}
for $n\gg0$, then 
\[
\rank_{\Zp}A^\vee(F_n)=O(1).
\]
\end{theo}
\begin{proof}
Proposition~\ref{prop:boundXloc} says that $e_\theta\cdot \Xloc(F_n)\otimes_{\Zp} \Qp=0$ for $n\gg0$. But
\[
\rank_{\Zp}\Xloc(F_n)-\rank_{\Zp}\Xloc(F_{n-1})=\dim_{\Qp}e_{\theta}\cdot \Xloc(F_n)\otimes_{\Zp} \Qp,
\]
where $\theta$ is any character of $\Gamma$ of conductor $p^{n+1}$. In particular, $\rank_{\Zp}\Xloc(F_n)=O(1)$. Lemma~\ref{lem:cokernel} now implies that
$\rank_{\Zp}\XX_p(A^\vee/F_n)=O(1)$ and our theorem follows from \eqref{eq:sel}.
\end{proof}
In other words, the key to show that the Mordell-Weil ranks of $A^\vee$ are bounded inside $F_\infty$ is to establish \eqref{eq:key}.
\subsection{Special cases}
If Conjecture~\ref{conj} holds, Lemma~\ref{lem:nonzero} tells us that $\col_{T,\uJ}(\bc)(\theta)\ne0$ if $\theta$ is a character whose conductor is sufficiently large. However, this is not enough to verify \eqref{eq:key} since we do not have an explicit description of $H_{\uI_0,\uJ,n}$ in the most general setting. In this section, we will show that when the matrices $C_{\vp,v}$ are explicit enough, it is possible to establish \eqref{eq:key} by calculating the $p$-adic valuations of $H_{\uI_0,\uJ,n}(\theta)$.

\subsubsection{Block anti-diagonal  matrices}We suppose in this section that for each $v$, we may find a basis of $\Dcrisv(T)$ such that the matrix $C_v$ is of the form $\left(
\begin{array}{c|c}
0&*\\ \hline
*&0
\end{array}
\right)$, where $*$ represents a $gf_v\times gf_v$ matrix defined over $\Zp$. This is the same as saying that $\vp(v_i)\notin\Fil^0\Dcrisv(T)$ and $\vp^2(v_i)\in\Fil^0\Dcrisv(T)$ for all $i\in \{1,\ldots gf_v\}$. It can be thought of as the analogue of $a_p=0$ for supersingular elliptic curves. In particular,  
\begin{equation} \label{eq:specialCn}
C_{v,n}=\left(
\begin{array}{c|c}
0&B_{v,1}\\ \hline
\Phi_{p^n}(1+X)B_{v,2}&0
\end{array}
\right)
\end{equation}
 for some invertible $gf_v\times gf_v$ matrices $B_{v,1}$ and $B_{v,2}$ that are defined over $\Zp$ with $\det(B_{v,1}B_{v,2})=1$ (since $\det C_{\vp,v}=p^{-gf_v}$).
For all $n\ge1$, we fix a primitive $p^n$-th root of unity $\zeta_{p^n}$  and we write $\epsilon_n=\zeta_{p^n}-1$. 

\begin{lemm}\label{lem:explicitHn}
Suppose that $C_v$ is block anti-diagonal for all $v$. Then, for all $n\ge1$, we have
\[
H_{v,n}(\zeta_{p^n}-1)=
\begin{cases}
\left(
\begin{array}{c|c}
0&\delta_n (B_{v,1}B_{v,2})^{(n-1)/2}B_{v,1}\\ \hline
0&0
\end{array}
\right)&\text{if $n$ odd,}\\
\left(
\begin{array}{c|c}
\delta_n(B_{v,1}B_{v,2})^{n/2}&0\\ \hline
0&0
\end{array}
\right)
&\text{if $n$ even.}
\end{cases}
\]
Here, the constant $\delta_n$ is given by
\[
\delta_n=\begin{cases}
\frac{\epsilon_1}{\epsilon_2}\cdot \frac{\epsilon_3}{\epsilon_4}\cdots\frac{\epsilon_{n-2}}{\epsilon_{n-1}}&\text{if $n$ odd,}\\
\frac{\epsilon_1}{\epsilon_2}\cdot \frac{\epsilon_3}{\epsilon_4}\cdots\frac{\epsilon_{n-1}}{\epsilon_{n}}&\text{if $n$ even.}
\end{cases}
\]
\end{lemm}
\begin{proof}
Thanks to \eqref{eq:specialCn}, we have explicitly
\[
H_{v,n}(\epsilon_n)=\left(
\begin{array}{c|c}
0&B_{v,1}\\ \hline
0&0
\end{array}
\right)\cdot\left(
\begin{array}{c|c}
0&B_{v,1}\\ \hline
\frac{\epsilon_1}{\epsilon_2} B_{v,2}&0
\end{array}
\right)\cdot\left(
\begin{array}{c|c}
0&B_{v,1}\\ \hline
\frac{\epsilon_2}{\epsilon_3}B_{v,2}&0
\end{array}
\right)\cdots \left(
\begin{array}{c|c}
0&B_{v,1}\\ \hline
\frac{\epsilon_{n-1}}{\epsilon_n}B_{v,2}&0
\end{array}
\right).
\]
Hence the result on multiplying out these matrices.
\end{proof}

 Recall that $\uI_0=(I_{v,0})_{v|p}$, where $I_{v,0}=\{1,\ldots,gf_v\}$. Let $\uI_1$ be the complement of $\uI_0$, that is $\uI_1=(I_{v,1})_{v|p}$, where $I_{v,1}=\{gf_v+1,\ldots 2gf_v\}$.

\begin{lemm}
Suppose that  $C_v$ is block anti-diagonal for all $v$ and that the Selmer groups $\sel_{\uI_0}(A/F_\infty)$ and $\sel_{\uI_1}(A/F_\infty)$ are both $\Zp[[\Gamma]]$-cotorsion. Then, \eqref{eq:key} holds whenever $n$ is sufficiently large.
\end{lemm}
\begin{proof}
Suppose that $\theta$ sends the fixed topological generator to $\zeta_{p^n}$. By Lemma~\ref{lem:explicitHn}, 
\[
H_{\uI_0,\uJ,n}(\theta)=0
\]
unless $\uJ=\uI_0$ and $n$ is even, or $\uJ=\uI_1$ and $n$ is odd. Let us write $\uJ_n=\uI_{\frac{1-(-1)^n}{2}}$. Then, the left-hand side of \eqref{eq:key} is equal to
\[
\left(H_{\uI_0,\uJ_n,n}\col_{T,\uJ_n}(\bc)\right)(\theta).
\]
Lemma~\ref{lem:explicitHn} says that $H_{\uI_0,\uJ_n,n}(\theta)$ is never zero. Our hypothesis on the $\uJ_n$-Selmer group and Lemma~\ref{lem:nonzero} tell us that $\left(\col_{T,\uJ_n}(\bc)\right)(\theta)\ne0$ when $n$ is sufficiently large. Hence  the result follows.
\end{proof}
If we combine this with Theorem~\ref{thm:key}, we deduce our first result on the Mordell-Weil ranks of $A^\vee$.
\begin{coro}\label{cor:antidiag}
Suppose that  $C_v$ is block anti-diagonal for all $v$ and that the Selmer groups $\sel_{\uI_0}(A/F_\infty)$ and $\sel_{\uI_1}(A/F_\infty)$ are both $\Zp[[\Gamma]]$-cotorsion. Then, \[
\rank_{\Zp}A^\vee(F_n)=O(1).
\]
\end{coro}

\subsubsection{Block anti-diagonal modulo $p$  matrices}
We suppose in this section that for each $v$, we may find a basis of $\Dcrisv(T)$ such that the matrix $C_v$ is of the form 
\[
C_{v}=\left(
\begin{array}{c|c}
pA_{v,1}&B_{v,1}\\ \hline
B_{v,2}&pA_{v,2}
\end{array}
\right),
\]
where $A_{v,1},A_{v,2},B_{v,1}$ $B_{v,2}$ are some $gf_v\times gf_v$ matrices over $\Zp$ with $\det(B_{v,1}),\det(B_{v,2})\in\Zp^\times$. In other words, $C_v$ is congruent to a block anti-diagonal matrix modulo $p$. Note that in the case of elliptic curves, given an basis of $v_1$ of $\Fil^0\Dcrisv(T)$, the pair $v_1,\vp(v_1)$ form a basis of $\Dcrisv(T)$ by Fontaine-Laffaille theory. The matrix of $\vp$ is given by $\begin{pmatrix}
0&-\frac{1}{p}\\1&\frac{a_p}{p}
\end{pmatrix}$. Thus, $C_v=\begin{pmatrix}
0&-1\\1&a_p
\end{pmatrix}$ is a block anti-diagonal matrix mod $p$ whenever $p|a_p$. We shall discuss in this next section that the same holds for abelian varieties of $\GL_2$-type in the next section.

From now on, for each prime $v$, $\ord_p$ denotes the normalized $p$-adic valuation on $\overline{F_v}$ with $\ord_p(p)=1$.
Recall from the previous section that $\uJ_n=\uI_{\frac{1-(-1)^n}{2}}$.
\begin{lemm}\label{lem:blockanti}
Suppose that $C_v$ is block anti-diagonal mod $p$ for all $v$. Then,
\[
\ord_p\left(H_{\uI_0,\uJ,n}(\zeta_{p^n}-1)\right)-\ord_p\left(H_{\uI_0,\uJ_n,n}(\zeta_{p^n}-1)\right)\ge 1-\frac{1}{p^2-1}
\]
for all $\uJ\ne \uJ_n$ and $n\ge1$.
\end{lemm}
\begin{proof}
By the calculations of Lemma~\ref{lem:explicitHn},
\[
H_{v,n}(\zeta_{p^n}-1)=
\begin{cases}
\left(
\begin{array}{c|c}
O(p)&\delta_n (B_{v,1}B_{v,2})^{(n-1)/2}B_{v,1}+O(p)\\ \hline
0&0
\end{array}
\right)&\text{if $n$ odd,}\\
\left(
\begin{array}{c|c}
\delta_n(B_{v,1}B_{v,2})^{n/2}+O(p)&O(p)\\ \hline
0&0
\end{array}
\right)
&\text{if $n$ even.}
\end{cases}
\]
Here, $O(p)$ denotes a matrix whose entries all have $p$-adic valuation $\ge1$. The explicit formula of $\delta_n$ given by 
\[
\ord_p(\delta_n)<\frac{1}{p^2}+\frac{1}{p^4}+\cdots=\frac{1}{p^2-1}<1.
\]
Therefore,
\[
H_{\uI_0,\uJ_n,n}(\zeta_{p^n}-1)=\det(\delta_n\star),
\]
where $\star$ is a matrix given by a product of $B_{v,1}$ and $B_{v,2}$. But the determinants of these  matrices are $p$-adic units, so
\[
\ord_p\left(H_{\uI_0,\uJ_n,n}(\zeta_{p^n}-1)\right)=g[F:\Q]\ord_p(\delta_n).
\]

If $\uJ\ne\uJ_n$, then there is at least one $v$ where the  $(I_{v,0},J_v)$-minor contains a column whose entries all have $p$-adic valuation $\ge1$. In the other columns, the entries all have $p$-adic valuation $\ge \ord_p(\delta_n)$. Therefore,
\[
\ord_p\left(H_{\uI_0,\uJ,n}(\zeta_{p^n}-1)\right)\ge 1+(g[F:\Q]-1)\ord_p(\delta_n).
\]
Hence, 
\[
\ord_p\left(H_{\uI_0,\uJ,n}(\zeta_{p^n}-1)\right)-\ord_p\left(H_{\uI_0,\uJ_n,n}(\zeta_{p^n}-1)\right)\ge 1-\ord_p(\delta_n)>1-\frac{1}{p^2-1},
\]
as required.
\end{proof}

\begin{lemm}\label{lem:specialcase2}
Suppose that $C_v$ is block anti-diagonal mod $p$ for all $v$ and that $\col_{T,\uI_0}(\bc)$ and $\col_{T,\uI_1}(\bc)$ are both non-zero. Furthermore, suppose that for all $\uJ$ such that  $\col_{T,\uJ}(\bc)\ne 0$, the $\mu$-invariant of  $\col_{T,\uJ}(\bc)$ is independent of $\uJ$. Then \eqref{eq:key} holds for $n\gg0$.
\end{lemm}
\begin{proof}
	Let $\mu$ be the common $\mu$-invariant among the non-zero $\col_{T,\uJ}(\bc)$ and write $\lambda_{\uJ}$ for the $\lambda$-invariant of $\col_{T,\uJ}(\bc)$. For $n\gg0$, by Weierstrass' preparation theorem~\cite[Theorem 7.3]{Washington97Book},
\[
\ord_p\left(\col_{T,\uJ}(\bc)(\zeta_{p^n}-1)\right)=\mu +\frac{\lambda_{\uJ}}{p^n-p^{n-1}}.
\]
Note that $\frac{\lambda_{\uJ}}{p^n-p^{n-1}}$ becomes arbitrarily small as $n\rightarrow \infty$. Therefore, on combining this with Lemma~\ref{lem:blockanti}, we deduce that
\[
0\ne\ord_p\left(H_{\uI_0,\uJ_n,n}\col_{T,\uJ_n}(\bc)(\zeta_{p^n}-1)\right)<\ord_p\left(H_{\uI_0,\uJ,n}\col_{T,\uJ}(\bc)(\zeta_{p^n}-1)\right)
\]
for all $\uJ\ne \uJ_n$ and $n\gg0$. Hence \eqref{eq:key} holds.
\end{proof}

\begin{coro}\label{cor:antidiagmodp}
Suppose that $C_v$ is block anti-diagonal mod $p$ for all $v$ and that $\sel_{\uI_0}(A/F_\infty)$ and $\sel_{\uI_1}(A/F_\infty)$ are both $\Zp[[\Gamma]]$-cotorsion. Furthermore, suppose that for all $\uJ$ such that  $\sel_{\uJ}(A/F_\infty)$ is $\Zp[[\Gamma]]$-cotorsion, the $\mu$-invariant of is $\sel_{\uJ}(A/F_\infty)^\vee$ independent of $\uJ$. Then,
\[
\rank_{\Zp}A^\vee(F_n)=O(1).
\]
\end{coro}
\begin{proof}
By Lemmas~\ref{lem:nonzero} and \ref{lem:specialcase2}, it is enough to establish the statement on the $\mu$-invariants. But \eqref{diag:signedPT} tells us that
\[
\mu\left(\col_{T,\uJ}(\bc)\right)-\mu\left(\sel_{\uJ}(A/F_\infty)^\vee\right)
\]
is independent of $\uJ$, so we are done.
\end{proof}

\begin{rema}
	In the case of elliptic curves with supersingular reduction at $p$, it is conjectured that the signed Selmer groups have zero $\mu$-invariants \cite[Conjecture 6.3]{Pollack} \cite[Conjecture 7.1]{PerrinRiou03}.
\end{rema}

\subsection{Abelian varieties of $\GL_2$-type}
We now assume that $A$ is an abelian variety defined over $\Q$ of $\GL_2$-type as defined in \cite{Ribet92}, that is, the algebra of $\Q$-endomorphisms of $A$ contains a number field $E$ of degree $[E:\Q]=\dim A$. We also assume that the ring of integers $\cO_E$ of $E$ is the ring of $\Q$-endomorphisms of $A$ and that $p$ is unramified in $E$.
In particular, the $p$-adic Tate module of $A$ splits into
\[
T\cong \bigoplus_{\p|p}T_\p(A),
\]
where the direct sum runs over all primes $\p$ of $E$ above $p$ and $T_\p(A)$ is a free $\cO_\p$-module of rank $2$ with $\cO_\p$ the completion of $\cO_E$ at $\p$.

Since $A$ is defined over $\Q$, we have $\Dcrisv(T_\p(A))=\Dcris(T_\p(A))\otimes \cO_v$, where $\Dcris(-)$ denotes the Dieudonn\'e module over $\Qp$. Therefore, it is suffisant to study the matrix of $\vp$ over $\Dcris(T_\p(A))$. The action of $\vp$ on $\Dcris(T_\p(A))$ is $\cO_\p$-linear turning $\Dcris(T_\p(A))$ into a rank-two filtered $\cO_\p$-module.

By considering the image of the Kummer map in $\coh^1(\Qp,T_\p(A))$, we see that the Hodge-Tate weights for this filtration are 0 and 1, each with multiplicity one.  Fontaine-Laffaille theory tells us that there exists an $\cO_\p$-basis of the form $\omega_{\p}$, $\vp(\omega_{\p})$, where $\omega_{\p}$ generates $\Fil^0\Dcris(T_\p(A))$. As $A$ is supersingular at $p$, the eigenvalues of $\vp$ are of the form $\zeta_i/\sqrt{p}$ $i=1,2$, where $\zeta_i$ is a root of unity.  But $p$ is unramified in $E_\p$. For the trace of $\vp$ to be an element of $E_\p$, $\zeta_1+\zeta_2$ must be an element of $p\cO_\p$. Therefore,  the matrix of $\vp$ with respect to the basis $\{\omega_\p,\vp(\omega_\p)\}$ is of the form $\begin{pmatrix}
0&\frac{b_\p}{p}\\
1&\frac{a_{\p}}{p}
\end{pmatrix}$
for some $a_{\p}\in p\cO_\p$ and $b_\p\in \cO_\p^\times$.
If we choose a $\Zp$-basis of $\cO_\p$, say $\{x_1,\ldots,x_{[E_\p:\Qp]}\}$, then this gives rise to a $\Zp$-basis of $\Dcris(T_\p(A))$, namely 
$$\left\{x_i\omega_{\p,v},x_i\vp(\omega_{\p,v}):i=1,\ldots,[E_\p:\Qp]\right\}.$$

Under this choice of bases, we see that the resulting matrix $C_v$ will be block anti-diagonal mod $p$ for all $v$. In particular, Corollary~\ref{cor:antidiagmodp}. Furthermore, if $a_{\p,v}=0$ for all $\p$ and $v$, then $C_v$ will even be block anti-diagonal. In this case, Corollary~\ref{cor:antidiag} applies.

\bibliographystyle{amsalpha}
\bibliography{references}

\providecommand{\bysame}{\leavevmode\hbox to3em{\hrulefill}\thinspace}
\providecommand{\MR}{\relax\ifhmode\unskip\space\fi MR }
\providecommand{\MRhref}[2]{%
  \href{http://www.ams.org/mathscinet-getitem?mr=#1}{#2}
}
\providecommand{\href}[2]{#2}
\begin{thebibliography}{Mum85}

\bibitem[Ber04]{BergerLimites}
Laurent Berger, \emph{Limites de repr\'esentations cristallines}, Compositio
  Mathematica \textbf{140} (2004), no.~6, 1473--1498.

\bibitem[BL17]{BLIntegral}
Kaz{\i}m B{\"u}y{\"u}kboduk and Antonio Lei, \emph{Integral {I}wasawa theory of
  {G}alois representations for non-ordinary primes}, Mathematische Zeitschrift
  \textbf{286} (2017), no.~1--2, 361--398.

\bibitem[Gre89]{GreenbergIwasawaRep}
Ralph Greenberg, \emph{Iwasawa theory for {$p$}-adic representations},
  Algebraic number theory, Advanced studies in pure mathematics, vol.~17,
  Academic press, Boston, MA, 1989, pp.~97--137.

\bibitem[Gre99]{GreenbergIwasawaEll}
\bysame, \emph{Iwasawa theory for elliptic curves}, Arithmetic theory of
  elliptic curves ({C}etraro, 1997), Lecture Notes in Mathematics, vol. 1716,
  Springer, Berlin, 1999, pp.~51--144.

\bibitem[Gre01]{Greenbergintro}
Ralph Greenberg, \emph{Introduction to {I}wasawa theory for elliptic curves},
  Arithmetic algebraic geometry ({P}ark {C}ity, {UT}, 1999), IAS/Park City
  Math. Ser., vol.~9, Amer. Math. Soc., Providence, RI, 2001, pp.~407--464.

\bibitem[Gre03]{GreenbergGaloisSelmer}
Ralph Greenberg, \emph{Galois theory for the {S}elmer group of an abelian
  variety}, Compositio Mathematica \textbf{136} (2003), no.~3, 255--297.

\bibitem[IK19]{ImKim}
Bo-Hae Im and Byoung~Du Kim, \emph{Ranks of rational points of the {J}acobian
  varieties of hyperelliptic curves}, J. Number Theory \textbf{195} (2019),
  23--50.

\bibitem[Ima75]{Imai}
Hideo Imai, \emph{A remark on the rational points of abelian varieties with
  values in cyclotomic {$\Zp$}-extensions}, Proceedings of the Japan Academy
  \textbf{51} (1975), 12--16.

\bibitem[Kat93]{Kato91}
Kazuya Kato, \emph{Lectures on the approach to {I}wasawa theory for
  {H}asse-{W}eil {$L$}-functions via {$\mathbf{B}_{\mathrm{dR}}$}. {P}art
  {I}.}, Arithmetic algebraic geometry ({T}rento, 1991), Lecture Notes in
  Mathematics, vol. 1553, Springer, Berlin, 1993, pp.~50--163.

\bibitem[Kat04]{Kato04}
\bysame, \emph{{$p$}-adic {H}odge theory and values of zeta functions of
  modular forms}, Ast\'erisque (2004), no.~295, ix, 117--290, Cohomologies
  {$p$}-adiques et applications arithm\'etiques. III.

\bibitem[Kim18]{Kimrank}
Byoung~Du Kim, \emph{Ranks of the rational points of abelian varieties over
  ramified fields, and {I}wasawa theory for primes with non-ordinary
  reduction}, Journal of Number Theory \textbf{183} (2018), 352--387.

\bibitem[Kob03]{Kobayashi}
Shin-Ichi Kobayashi, \emph{Iwasawa theory for elliptic curves at supersingular
  primes}, Inventiones Mathematicae \textbf{152} (2003), 1--36.

\bibitem[LLZ10]{LLZWach}
Antonio Lei, David Loeffler, and Sarah~Livia Zerbes, \emph{Wach modules and
  {I}wasawa theory for modular forms}, Asian Journal of Mathematics \textbf{14}
  (2010), no.~4, 475--528.

\bibitem[LZ14]{LLZp2}
David Loeffler and Sarah~Livia Zerbes, \emph{Iwasawa theory and {$p$}-adic
  {$L$}-functions over {$\Zp^2$}-extensions}, International Journal of Number
  Theory \textbf{10} (2014), no.~8, 2045--2095.

\bibitem[Maz72]{Mazur72}
Barry Mazur, \emph{Rational points of abelian varieties with values in towers
  of number fields}, Inventiones Mathematicae \textbf{18} (1972), 183--266.

\bibitem[Mum85]{MumfordAbVar}
David Mumford, \emph{Abelian varieties}, Tata institutes of Fundamental
  Research Studies in Mathematics, vol.~5, Published for the Tata Institute of
  Fundamental Research, Bombay; Oxford University Press, 1985, With appendices
  by C. P. Ramanujam and Yuri Manin.

\bibitem[Pol03]{Pollack}
Robert Pollack, \emph{On the {$p$}-adic {$L$}-function of a modular form at a
  supersingular prime}, Duke Mathematical Journal \textbf{118} (2003), no.~3,
  523--558.

\bibitem[PR90]{PR90}
Bernadette Perrin-Riou, \emph{Th{'e}orie d'{I}wasawa {$p$}-adique locale et
  globale}, Inventiones Mathematicae \textbf{99} (1990), no.~2, 246--292.

\bibitem[PR94]{PR94}
\bysame, \emph{Th\'eorie d'{I}wasawa des repr\'esentations {$p$}-adiques sur un
  corps local}, Inventiones Mathematicae \textbf{115} (1994), no.~1, 81--161,
  With an appendix by Jean-Marc Fontaine.

\bibitem[PR00]{PRLivre}
\bysame, \emph{{$p$}-adic {$L$}-functions and {$p$}-adic representations},
  SMF/AMS Texts and Monographs, vol.~3, American Mathematical Society,
  Providence, RI; Soci\'et\'e Math\'ematique de France, Paris, 2000, Translate
  from the 1995 French original by Leila Schneps and revised by the author.

\bibitem[PR03]{PerrinRiou03}
Bernadette Perrin-Riou, \emph{Arithm\'{e}tique des courbes elliptiques \`a
  r\'{e}duction supersinguli\`ere en {$p$}}, Experiment. Math. \textbf{12}
  (2003), no.~2, 155--186.

\bibitem[Rib92]{Ribet92}
Kenneth~A. Ribet, \emph{Abelian varieties over {${\bf Q}$} and modular forms},
  Algebra and topology 1992 ({T}aej\u{o}n), Korea Adv. Inst. Sci. Tech.,
  Taej\u{o}n, 1992, pp.~53--79.

\bibitem[Roh84]{Rohrlich84}
David Rohrlich, \emph{On {$L$}-functions of elliptic curves and cyclotomic
  towers}, Inventiones Mathematicae \textbf{75} (1984), no.~3, 409--423.

\bibitem[Spr12]{Sprung}
Florian E.~Ito Sprung, \emph{Iwasawa theory for elliptic curves at
  supersingular primes: a pair of main conjectures}, Journal of Number Theory
  \textbf{132} (2012), no.~7, 1483--1506.

\bibitem[Was97]{Washington97Book}
Lawrence Washington, \emph{Introduction to cyclotomic fields}, second ed.,
  Graduate Texts in Mathematics, vol.~83, Springer-Verlag, New York, 1997.

\end{thebibliography}
\end{document}